\documentclass[11pt]{amsart}
\usepackage{latexsym}
\usepackage{amssymb}
\usepackage{amsthm}
\usepackage{tikz}
\usepackage{amscd}
\usepackage[all,cmtip]{xy}

\makeatletter
\@namedef{subjclassname@2010}{\textup{2010} Mathematics Subject Classification}
\makeatother
\newtheorem{theorem}{Theorem}[section]
\newtheorem{lemma}[theorem]{Lemma}
\newtheorem{proposition}[theorem]{Proposition}

\newtheorem{corollary}[theorem]{Corollary}

\newcommand{\frk}{\operatorname{frk}}
\newtheorem*{conjecture1}{Conjecture I}
\newtheorem*{conjecture2}{Conjecture II}
\newtheorem*{question}{Question}

\begin{document}
\title{Free 2-rank of symmetry of products of Milnor manifolds}
\author{Mahender Singh}
\address{Indian Institute of Science Education and Research (IISER) Mohali, Sector 81, Knowledge City, S A S Nagar (Mohali), Post Office Manauli, Punjab 140306, India.}
\email{mahender@iisermohali.ac.in}

\begin{abstract}
A real Milnor manifold is the non-singular hypersurface of degree $(1,1)$ in the product of two real projective spaces. These manifolds were introduced by Milnor to give generators for the unoriented cobordism algebra, and they admit free actions by elementary abelian 2-groups. In this paper, we obtain some results on the free 2-rank of symmetry of products of finitely many real Milnor manifolds under the assumption that the induced action on mod 2 cohomology is trivial. Similar results are obtained for complex Milnor manifolds which are defined analogously. Here the free 2-rank of symmetry of a topological space is the maximal rank of an elementary abelian 2-group which acts freely on that space.
\end{abstract}

\maketitle

\section{Introduction}
One of the basic problems in the theory of transformation groups is to determine the structure of a group which acts in a specific way on a given topological space. The problem of determining finite groups that can act freely on spheres has been of special interest. A classical result of Smith \cite{Smith} says that; if a group acts freely on a sphere, then all its abelian subgroups are cyclic. Conversely, Swan \cite{Swan} proved that any group satisfying this condition acts freely on a finite complex with the homotopy type of a sphere. Notice that a finite abelian group is cyclic if and only if it does not contain a subgroup of the form $\mathbb{Z}/p \oplus \mathbb{Z}/p$ for any prime $p$. Thus $\mathbb{Z}/p \oplus \mathbb{Z}/p$ cannot act freely on a sphere. However, $\mathbb{Z}/p$ does act freely on a sphere, where the sphere must be odd dimensional for odd $p$. These results motivated the concept of free $p$-rank of symmetry of a topological space $X$ for a prime $p$, introduced in \cite{Adem1}, and defined as
$$\frk_p(X)= \max\{r ~|~ (\mathbb{Z}/p)^r~\textrm{acts freely on}~ X\}.$$
Determining the free $p$-rank of symmetry of a topological space is an interesting problem and has been considered for many spaces. In view of the Smith theory, we have
\begin{displaymath}
\frk_p(\mathbb{S}^n) = \left\{ \begin{array}{ll}
1 & \textrm{if $n$ is odd and $p$ is arbitrary}\\
1 & \textrm{if $n$ is even and $p=2$}\\
0 & \textrm{if $n$ is even and $p>2$.}
\end{array} \right.
\end{displaymath}
The problem of extending this result to products of finitely many spheres has been of great interest to many topologists. Oliver \cite{Oliver} proved that every finite group acts freely on a product of spheres. For a finite group $G$, define 
$$h(G) = \min\{s \mid G ~\textrm{acts freely on}~ \mathbb{S}^{n_1} \times \cdots \times \mathbb{S}^{n_s} \}~\textrm{and}$$
$$r(G)=\max \{t \mid (\mathbb{Z}/p)^t \leq G~\textrm{for some prime}~ p\}.$$
Attempts to extend Smith's result to products of finitely many spheres led to the following conjecture by Benson and Carlson \cite{Benson}.
\begin{conjecture1}
$h(G)=r(G)$ for each finite group $G$.
\end{conjecture1}
Conjecture I implies the following well known conjecture \cite{Adem1, Adem2, Carlsson1} regarding free actions of elementary abelian $p$-groups on products of spheres.
\begin{conjecture2}
$\frk_p(\mathbb{S}^{2n_1+1} \times \cdots \times \mathbb{S}^{2n_k+1} )=k$ for each prime $p$ and integer $k \geq 0$.
\end{conjecture2}

For a single sphere it is simply the result of Smith \cite{Smith}. For products of two spheres it was already proved by Conner \cite{Conner1}, and for products of three spheres it was proved to be true by Heller \cite{Heller}. Carlsson \cite{Carlsson1, Carlsson2} proved the conjecture for products of equidimensional spheres under the assumption of trivial induced action on cohomology. Adem and Browder \cite{Adem1} proved that $\frk_p((\mathbb{S}^n)^k)=k$ with the only remaining cases as $p=2$ and $n=1,3,7$. Some time later,  Yal\c c\i n \cite{Yalcin2} proved that $\frk_2((\mathbb{S}^1)^k)=k$. For an integer $n$, let
\begin{displaymath}
\eta(n) = \left\{ \begin{array}{ll}
0 & \textrm{if $n$ is even}\\
1 & \textrm{if $n$ is odd.}
\end{array} \right.
\end{displaymath}
Then the most  comprehensive result is due to Hanke \cite{Hanke} who proved that; if $p > 3(n_1+\cdots+n_k)$, then
$$\frk_p(\mathbb{S}^{n_1} \times \cdots \times \mathbb{S}^{n_k}) = \eta(n_1)+ \cdots+\eta(n_k).$$
In a very recent work \cite{Okutan}, Okutan and Yal\c c\i n proved the Conjecture II in the case where the dimensions $\{n_i\}$ are high compared to all the differences $|n_i - n_j |$ between the dimensions.

Recall that $\mathbb{Z}/2$ is the only finite group that can act freely on an even dimensional sphere. This result was extended to products of finitely many even dimensional spheres by Cusick \cite{Cusick2, Cusick4}. He proved that; if $G$ is a finite group acting freely on $\mathbb{S}^{2n_1} \times \cdots \times \mathbb{S}^{2n_k}$ with trivial induced action on mod 2 cohomology, then $G\cong (\mathbb{Z}/2)^r$ for some $r\leq k$. In particular, when the induced action on mod 2 cohomology is trivial, then $\frk_2(\mathbb{S}^{2n_1} \times \cdots \times \mathbb{S}^{2n_k}) = k$.

Although a lot of work has been done for products of spheres, free $p$-rank of symmetry of many other interesting spaces is still not known. An immediate extension of the problem from spheres and their products is to consider spherical space forms and their products. Let $p$ be an odd prime. A lens space $L_p^{2n-1}$ is an odd dimensional spherical space form obtained as the quotient of the standard $\mathbb{Z}/p$ action on $\mathbb{S}^{2n-1}$. Allday \cite{Allday2} conjectured that $$\frk_p(L_p^{2n_1-1} \times \cdots \times L_p^{2n_k-1}) = k.$$ The equidimensional case of the conjecture was proved by Yal\c c\i n in \cite{Yalcin1}. The general case of the conjecture seems still open.

The problem of computing the free 2-rank of symmetry of products of projective spaces was considered by Cusick. It is known that $\mathbb{C}P^n$ admits a free action by a finite group if and only if $n$ is odd, in which case the only possible group is $\mathbb{Z}/2$. Cusick \cite{Cusick5} proved that
$$\frk_2(\mathbb{C}P^{n_1} \times \cdots \times \mathbb{C}P^{n_k})= \eta(n_1)+ \cdots+\eta(n_k).$$

For an integer $n$, let
\begin{displaymath}
\theta(n) = \left\{ \begin{array}{ll}
0 & \textrm{if $n$ is even}\\
1 & \textrm{if $n \equiv 1 \mod 4$}\\
2 & \textrm{if $n \equiv 3 \mod 4$.}
\end{array} \right.
\end{displaymath}
Cusick \cite{Cusick1} investigated the real case and conjectured that; if the induced action on mod 2 cohomology is trivial, then $$\frk_2(\mathbb{R}P^{n_1} \times \cdots \times \mathbb{R}P^{n_k})= \theta(n_1)+ \cdots+\theta(n_k).$$
He proved the conjecture when $n_i \not\equiv 3 \mod 4$ for each $1 \leq i \leq k$. Adem and Yal\c c\i n \cite{Adem3} proved the conjecture for  products of equidimensional real projective spaces without the assumption of trivial induced action on mod 2 cohomology. Later, Yal\c c\i n \cite{Yalcin2} proved the conjecture when $n_i$ is odd for each $1 \leq i \leq k$. The general case of the conjecture is still open.

A product of two projective spaces can be considered as the total space of a trivial projective space bundle over a projective space. It is an interesting question to determine the free rank of symmetry of the total space of a twisted projective space bundle over a projective space. Milnor manifolds are fundamental examples of such spaces. These manifolds were introduced by Milnor \cite{Milnor} in search for generators for the unoriented cobordism algebra, and are non-singular hypersurfaces of degree $(1,1)$ in the product of two projective spaces (see section \ref{sec2} for detailed definitions). The purpose of this paper is to obtain some results regarding the free 2-rank of symmetry of products of finitely many Milnor manifolds. We will consider both real and complex Milnor manifolds. We will prove our results by adopting Cusick's method which depends on some results of Carlsson.

Let $\mathbb{R}H_{r,s}$ and $\mathbb{C}H_{r,s}$ denote the real and the complex Milnor manifold respectively (see section \ref{sec2} for notations). Let $X \simeq_2 Y$ mean that $X$ and $Y$ are topological spaces having isomorphic mod 2 cohomology algebra. Then the main results are as follows.

\begin{theorem}\label{thm1.1}
Let $(\mathbb{Z}/2)^r$ act freely on a finite dimensional CW-complex $X \simeq_2 \prod_{i=1}^n \mathbb{R}H_{r_i,s_i}$ with $1 \leq s_i \leq r_i$ for each $1 \leq i \leq n$. Suppose that the induced action on mod $2$ cohomology of $X$ is trivial. Then
\begin{enumerate}
\item $r \leq 2\big(\eta(s_1)+\eta(r_1)+\cdots+\eta(s_n)+\eta(r_n)\big)$
\item $r \leq \big(\eta(s_1)+\eta(r_1)+\cdots+\eta(s_n)+\eta(r_n)\big)$ if $s_i,r_i \not\equiv 3 \mod 4$ for each $1 \leq i \leq n$.
\end{enumerate}
\end{theorem}

\begin{theorem}\label{thm1.2}
Let $(\mathbb{Z}/2)^r$ act freely on a finite dimensional CW-complex $X \simeq_2 \prod_{i=1}^n \mathbb{C}H_{r_i,s_i}$ with $1 \leq s_i \leq r_i$ for each $1 \leq i \leq n$. Suppose that the induced action on mod $2$ cohomology of $X$ is trivial. Then
\begin{enumerate}
\item $r \leq 3\big(\eta(s_1)+\eta(r_1)+\cdots+\eta(s_n)+\eta(r_n)\big)$
\item $r \leq (\eta(s_1)+\eta(r_1)+\cdots+\eta(s_n)+\eta(r_n)\big)$ if $s_i,r_i \not\equiv 3 \mod 4$ for each $1 \leq i \leq n$.
\end{enumerate}
\end{theorem}

\section{Definition and cohomology of Milnor manifolds}\label{sec2}
Milnor manifolds were introduced by Milnor \cite{Milnor} in search for generators for the unoriented cobordism algebra. Let $r$ and $s$ be integers such that $0 \leq s \leq r$. A real Milnor manifold, denoted by $\mathbb{R}H_{r,s}$, is the non-singular hypersurface of degree $(1,1)$ in the product $\mathbb{R}P^r \times \mathbb{R}P^s$. It is a $(s+r-1)$ dimensional closed
smooth manifold, and can also be described in terms of homogeneous coordinates of real projective spaces as $$\mathbb{R}H_{r,s}=\Big\{\big([x_0,\dots,x_r], [y_0,\dots,y_s]\big) \in \mathbb{R}P^r \times \mathbb{R}P^s|~ x_0y_0+\dots+x_sy_s=0\Big\}.$$
Equivalently, a real Milnor manifold can be defined as the total space of the fiber bundle $$\mathbb{R}P^{r-1} \stackrel{i}{\hookrightarrow} \mathbb{R}H_{r,s} \stackrel{\pi}{\longrightarrow} \mathbb{R}P^{s}.$$ This is actually the projectivization of the vector bundle $$\mathbb{R}^r \hookrightarrow E^{\perp} \longrightarrow \mathbb{R}P^{s},$$ which is the orthogonal complement in $\mathbb{R}P^s \times \mathbb{R}^{r+1}$ of the canonical line bundle $$\mathbb{R} \hookrightarrow E \longrightarrow \mathbb{R}P^{s},$$ where $E=\Big\{\big([x],y\big) \in \mathbb{R}P^s \times \mathbb{R}^{r+1}|~ y \in [x] \Big\}.$

These manifolds are important as the unoriented cobordism algebra of smooth manifolds is generated by the cobordism classes of $\mathbb{R}P^k$ and $\mathbb{R}H_{r,s}$ \cite[Lemma 1]{Milnor}. Therefore, determining their various invariants is an important problem.

In a similar way, a complex Milnor manifold, denoted by $\mathbb{C}H_{r,s}$, is the non-singular hypersurface of degree $(1,1)$ in the product $\mathbb{C}P^r \times \mathbb{C}P^s$. It is a $2(s+r-1)$ dimensional closed
smooth manifold, and can also be described in terms of homogeneous coordinates as
$$\mathbb{C}H_{r,s}=\Big\{\big([z_0,\dots,z_r],[w_0,\dots,w_s]\big)\in \mathbb{C}P^r \times \mathbb{C}P^s|~ z_0\overline{w}_0+\dots+z_s\overline{w}_s=0\Big\}.$$ Equivalently, as in the real case, a complex Milnor manifold can be defined as the total space of the fiber bundle $$\mathbb{C}P^{r-1} \stackrel{i}{\hookrightarrow} \mathbb{C}H_{r,s} \stackrel{\pi}{\longrightarrow} \mathbb{C}P^{s}.$$ It is known due to Conner and Floyd \cite[p.63]{Conner2} that $\mathbb{C}H_{r,s}$ is unoriented cobordant to $\mathbb{R}H_{r,s} \times \mathbb{R}H_{r,s}$.

These manifolds have been well studied in the past. See for example \cite{Galvez, Kamata, Mukerjee} for recent results. Their cohomology algebra is also well known, and we will use it in the proofs of our main theorems.

\begin{lemma}\label{lem2.1}\cite{Bukhshtaber, Mukerjee}
Let $0 \leq s \leq r$. Then the mod $2$ cohomology algebra of a Milnor manifold is given as follows:
\begin{enumerate}
\item $H^*(\mathbb{R}H_{r,s}; \mathbb{Z}/2) \cong \mathbb{Z}/2[a,b]/\langle a^{s+1}, b^r+ab^{r-1}+ \cdots+a^sb^{r-s} \rangle$,\\
where $a$ and $b$ are homogeneous elements of degree one each.
\item $H^*(\mathbb{C}H_{r,s}; \mathbb{Z}/2) \cong \mathbb{Z}/2[g,h]/\langle g^{s+1},h^r+gh^{r-1}+\cdots+g^sh^{r-s} \rangle$,\\
where $g$ and $h$ are homogeneous elements of degree two each.
\end{enumerate}
\end{lemma}

Note that $\mathbb{R}H_{r,0}= \mathbb{R}P^{r-1}$ and $\mathbb{C}H_{r,0}= \mathbb{C}P^{r-1}$. Since the free 2-rank of symmetry of products of projective spaces has already been considered by Cusick \cite{Cusick1}, we henceforth assume that $1 \leq s \leq r$.
\bigskip

\section{Examples of free actions of elementary abelian 2-groups}\label{sec3}
Just like projective spaces, Milnor manifolds also admit free actions by elementary abelian 2-groups.

\subsection{The real case}
First we construct free actions on $\mathbb{R}H_{r,s}$ for various values of $s$ and $r$. 

\subsubsection{When $s=r$}\label{sec3.1.1}
The involution on $\mathbb{R}P^s \times \mathbb{R}P^s$ given by $$\big([x_0,\dots,x_s], [y_0,\dots,y_s]\big) \longmapsto \big([y_0,\dots,y_s], [x_0,\dots,x_s]\big)$$ restricts to a free involution $A:\mathbb{R}H_{s,s} \longrightarrow \mathbb{R}H_{s,s}$.

\subsubsection{When $s,r \equiv 1 \mod 4$}\label{sec3.1.2}
Let $n=s,r$ and $S_1:\mathbb{R}P^n \longrightarrow \mathbb{R}P^n$ be the free involution given by $$[x_0,x_1,\dots, x_{n-1},x_n] \longmapsto [-x_1,x_0,\dots, -x_n,x_{n-1}].$$ Then the restriction of $S_1 \times S_1:\mathbb{R}P^r \times \mathbb{R}P^s \longrightarrow \mathbb{R}P^r \times \mathbb{R}P^s$ on $\mathbb{R}H_{r,s}$ gives a free involution $$A_1:\mathbb{R}H_{r,s} \longrightarrow \mathbb{R}H_{r,s}.$$

\subsubsection{When $s,r \equiv 3 \mod 4$}\label{sec3.1.3}
First notice that the free involution $A_1$ is also defined in this case. Let $n=s,r$ and $S_2:\mathbb{R}P^n \longrightarrow \mathbb{R}P^n$ be the free involution given by $$[x_0,x_1,x_2,x_3,\dots,x_{r-3},x_{r-2},x_{r-1},x_r]\longmapsto [-x_2,x_3,x_0,-x_1,\dots,-x_{r-1},x_r,x_{r-3},-x_{r-2}].$$ Then the restriction of $S_2 \times S_2:\mathbb{R}P^r \times \mathbb{R}P^s \longrightarrow \mathbb{R}P^r \times \mathbb{R}P^s$ on $\mathbb{R}H_{r,s}$ gives a free involution $$A_2:\mathbb{R}H_{r,s} \longrightarrow \mathbb{R}H_{r,s}.$$ Notice that $A_1 \neq A_2$, $A_1A_2=A_2A_1$ and  $A_1A_2$ acts freely on $\mathbb{R}H_{r,s}$. Therefore, $(\mathbb{Z}/2)^2$ acts freely on $\mathbb{R}H_{r,s}$ when $s,r \equiv 3 \mod 4$.

\subsection{The complex case}\label{sec3.2}
Now we construct free actions on $\mathbb{C}H_{r,s}$ for various values of $s$ and $r$. 

\subsubsection{When $s=r$}\label{sec3.2.1}
The involution on $\mathbb{C}P^s \times \mathbb{C}P^s$ given by $$\big([z_0,\dots,z_s], [w_0,\dots,w_s]\big) \longmapsto \big([w_0,\dots,w_s], [z_0,\dots,z_s] \big)$$ restricts to a free involution $B:\mathbb{C}H_{s,s} \longrightarrow \mathbb{C}H_{s,s}$.

\subsubsection{When both $s$ and $r$ are odd}\label{sec3.2.2}
It is known that $\mathbb{C}P^n$ admits a free action by a finite group if and only if $n$ is odd and in that case the only possible group is $\mathbb{Z}/2$. Let $n=s,r$ and $T_1:\mathbb{C}P^n  \longrightarrow \mathbb{C}P^n$ be the free involution given by $$[z_0,z_1,\dots, z_{n-1},z_n] \longmapsto [-\overline{z}_1,\overline{z}_0,\dots, -\overline{z}_n,\overline{z}_{n-1}].$$ Then the restriction of $T_1 \times T_1:\mathbb{C}P^r \times \mathbb{C}P^s \longrightarrow \mathbb{C}P^r \times \mathbb{C}P^s$ on $\mathbb{C}H_{r,s}$ gives a free involution $$B_1:\mathbb{C}H_{r,s} \longrightarrow \mathbb{C}H_{r,s}.$$

\subsubsection{When $s$ is odd and $r$ is even}\label{sec3.2.3}
Let $T_2:\mathbb{C}P^r  \longrightarrow \mathbb{C}P^r$ be the involution given by $$[z_0,z_1,\dots,z_{r-2}, z_{r-1},z_r] \longmapsto [-\overline{z}_1,\overline{z}_0,\dots, -\overline{z}_{r-1},\overline{z}_{r-2},\iota z_r],$$  where $\iota^2=-1$. Notice that $T_2$ is not a free involution. But the restriction of $T_2 \times T_1:\mathbb{C}P^r \times \mathbb{C}P^s \longrightarrow \mathbb{C}P^r \times \mathbb{C}P^s$ on $\mathbb{C}H_{r,s}$ gives a free involution $$B_2:\mathbb{C}H_{r,s} \longrightarrow \mathbb{C}H_{r,s}.$$

\section{Preliminary results}
Here we recall some facts that we will use in this paper. We refer the reader to \cite{Allday1, Borel, Mccleary} for details on the cohomology theory of transformation groups and spectral sequences. We refer to \cite{Mosher} for basic properties of Steenrod algebra. Throughout we will use cohomology with $\mathbb{Z}/2$ coefficients and we will suppress it from cohomology notation. All spaces under consideration will be finite dimensional CW-complexes. More generally, we can also consider paracompact spaces of finite cohomological dimension or finitistic spaces (which includes paracompact spaces of finite covering dimension and compact Hausdorff spaces).

All group actions under consideration are assumed to be continuous. Let $G$ be a finite group acting on a space $X$, and let $$G \hookrightarrow E_G \longrightarrow B_G$$ be the universal principal $G$-bundle. Let $$X_G=(X \times E_G) /G$$ be the orbit space of the diagonal action on $X \times E_G$. Then the projection $$X \times E_G \longrightarrow E_G$$ is $G$-equivariant, and gives a fibration $$X\hookrightarrow X_G \longrightarrow B_G$$ called the Borel fibration \cite[Chapter IV]{Borel}. Recall that for $G= (\mathbb{Z}/2)^r$, we have
$$B_G= \underbrace{\mathbb{R}P^{\infty} \times \cdots \times \mathbb{R}P^{\infty}}_{r~\textrm{times}}$$
and hence
$$H^*(B_G; \mathbb{Z}/2)= \mathbb{Z}/2[\alpha_1,\dots, \alpha_r],$$
where $\alpha_i$ is a homogeneous element of degree one for each $1 \leq i \leq r$.

We will exploit the Leray-Serre spectral sequence associated to a fibration as given by the following theorem.

\begin{theorem}\cite[Theorem 5.2]{Mccleary}\label{thm4.1}
Let $X \hookrightarrow E \longrightarrow B$ be a fibration. Then there is a first quadrant spectral sequence of algebras $\{E_r^{*,*}, d_r \}$, converging to $H^*(E)$ as an algebra, with $$ E_2^{k,l}= H^k\big(B; \mathcal{H}^l(X)\big),$$ the cohomology of the base $B$ with local coefficients in the cohomology of the fiber $X$.
\end{theorem}

The product in $E_{r+1}^{*,*}$ is induced by the product in $E_r^{*,*}$ and the differentials are derivations. Further, there is an isomorphism of graded commutative algebra
$$H^*(E) \cong \textrm{Tot}E_{\infty}^{*,*},$$
where $\textrm{Tot}E_{\infty}^{*,*}$ is the total complex of $E_{\infty}^{*,*}$.

\begin{theorem}\cite[Theorem 5.9]{Mccleary}\label{thm4.2}
Let $X\stackrel{i}{\hookrightarrow} E \stackrel{\pi}{\longrightarrow} B$ be a fibration. Suppose that the system of local coefficients on $B$ is simple, then the edge homomorphisms
$$H^k(B)=E_2^{k,0} \longrightarrow E_3^{k,0}\longrightarrow \cdots  \longrightarrow E_k^{k,0} \longrightarrow E_{k+1}^{k,0}=E_{\infty}^{k,0}\subset H^k(E)~\textrm{and}$$
$$H^l(E) \longrightarrow E_{\infty}^{0,l}= E_{l+1}^{0,l} \subset E_{l}^{0,l} \subset \cdots \subset E_2^{0,l}= H^l(X)$$
are the homomorphisms $$\pi^*: H^k(B) \longrightarrow H^k(E)~ and ~ i^*: H^l(E) \longrightarrow H^l(X).$$
\end{theorem}

Next we recall some results regarding elementary abelian 2-group actions on finite dimensional CW-complexes.

\begin{theorem}\cite[Theorem 3.10.4]{Allday1}\label{thm4.3}
Let $G=(\mathbb{Z}/2)^r$ act freely on a finite dimensional CW-complex $X$. Suppose that $\sum_{i\geq0} \textrm{rk}\big( H^i(X) \big)< \infty$ and the induced action on $H^*(X)$ is trivial, then the Leray-Serre spectral sequence associated to $X \hookrightarrow X_G \longrightarrow B_G$ does not degenerate at the $E_2$ term.
\end{theorem}

\begin{proposition}\cite[Proposition 3.10.9 and Lemma 3.10.16]{Allday1}\label{prop4.4}
Let $G =(\mathbb{Z}/2)^r$ act freely on a finite dimensional CW-complex $X$. Then $H^*(X/G) \cong H^*(X_G)$. Further, if $H^i(X)=0$ for all $i>n$, then $H^i(X_G)=0$ for all $i>n$.
\end{proposition}

We will also use the following results regarding non-trivial common zeros of a system of homogeneous polynomials.

\begin{proposition}\cite[Proposition 1]{Carlsson1}\label{prop4.5}
Let $G= (\mathbb{Z}/2)^r$ and $f_1,\dots,f_n$ be elements of $H^m(B_G)$ regarded as homogeneous polynomials of degree $m$ in $r$ variables. Then they have a non-trivial common zero in $(\mathbb{Z}/2)^r$ if and only if there is a subgroup inclusion $j:\mathbb{Z}/2 \hookrightarrow (\mathbb{Z}/2)^r$ such that $j^*(f_i)=0$ for each $1 \leq i \leq n$, where $j^*:H^*(B_G) \to H^*(B_{\mathbb{Z}/2})$ is the induced map on cohomology.
\end{proposition}

\begin{proposition}\cite[Proposition 4]{Carlsson1}\label{prop4.6}
Let $G= (\mathbb{Z}/2)^r$ and $f_1,\dots,f_n$ be elements of $H^*(B_G)$ regarded as homogeneous polynomials in $r$ variables. Suppose that the ideal $\langle f_1,\dots,f_n \rangle$ is invariant under the action of the Steenrod algebra. Then they have a non-trivial common zero in $(\mathbb{Z}/2)^r$ if $r>n$.
\end{proposition}

\begin{proposition}\cite{Greenberg}\label{prop4.7}
Let $f_1,\dots,f_n$ be homogeneous polynomials of degree $m$ in $r$ variables with coefficients in $\mathbb{Z}/2$. Then they have a non-trivial common zero in $(\mathbb{Z}/2)^r$ if $r>mn$.
\end{proposition}

\section{Induced action on cohomology}
Given a continuous map of topological spaces, determining the induced map on cohomology is a difficult problem in general, even for nice spaces such as spheres. In this section, we show that there are involutions on Milnor manifolds for which the induced action on mod 2 cohomology is non-trivial. First we consider the real case.

\begin{proposition}\label{prop5.1}
Let $A: \mathbb{R}H_{s,s} \longrightarrow \mathbb{R}H_{s,s}$ be the free involution given by
$$A\big([x_0,\dots,x_s], [y_0,\dots,y_s]\big) \longmapsto \big([y_0,\dots,y_s], [x_0,\dots,x_s]\big).$$
Then $A^*: H^*(\mathbb{R}H_{s,s}) \longrightarrow H^*(\mathbb{R}H_{s,s})$ is non-trivial.
\end{proposition}
\begin{proof}
By lemma \ref{lem2.1} $$H^*(\mathbb{R}H_{s,s}; \mathbb{Z}/2) \cong \mathbb{Z}/2[a,b]/\langle a^{s+1}, b^s+ab^{s-1}+\cdots+a^s \rangle,$$ where $a$ and $b$ are homogeneous elements of degree one each. By the K\"{u}nneth formula $$H_1(\mathbb{R}P^s \times \mathbb{R}P^s) \cong H_1(\mathbb{R}P^s) \oplus H_1(\mathbb{R}P^s).$$ Let $\sigma=(\sigma_1,\sigma_2):\Delta^1 \to \mathbb{R}P^s \times \mathbb{R}P^s$ be a singular 1-simplex and $$A: \mathbb{R}P^s \times \mathbb{R}P^s \to \mathbb{R}P^s \times \mathbb{R}P^s$$ denote the same involution. Then $A_{*}\big([\sigma]\big)=[A \circ \sigma]=[(\sigma_2,\sigma_1)]$. This shows that if $a_1$ and $b_1$ are generators of $H_1(\mathbb{R}P^s) \oplus H_1(\mathbb{R}P^s)$, then $A_{*}(a_1)=b_1$ and hence the induced action on $H_1(\mathbb{R}P^s \times \mathbb{R}P^s)$ is non-trivial.

Further we have $H^1(\mathbb{R}P^s \times \mathbb{R}P^s) \cong H^1(\mathbb{R}P^s) \oplus H^1(\mathbb{R}P^s)$. If $f \in H^1(\mathbb{R}P^s \times \mathbb{R}P^s)$, then $A^*(f)= f \circ A_{*}$. In particular, $A^*(f)(a_1)=  f \circ A_{*}(a_1)= f(b_1)$. Choosing $f$ such that $f(a_1)\neq f(b_1)$, we see that $A^*$ acts non-trivially on cohomology. Thus if $a_2$ and $b_2$ are generators of $H^1(\mathbb{R}P^s) \oplus H^1(\mathbb{R}P^s)$, then $A^{*}(a_2)=b_2$.

Recall that the real Milnor manifold is also given by the fiber bundle
$$\mathbb{R}P^{s-1} \stackrel{i}{\hookrightarrow} \mathbb{R}H_{s,s} \stackrel{\pi}{\longrightarrow} \mathbb{R}P^{s}.$$ Let $\mathbb{R}H_{s,s} \stackrel{j_1}{\hookrightarrow} \mathbb{R}P^s \times \mathbb{R}P^s$ be the canonical inclusion and
$$\mathbb{R}P^s \stackrel{j_2}{\hookrightarrow} \mathbb{R}P^s \times \mathbb{R}P^s  \stackrel{pr_1}{\longrightarrow} \mathbb{R}P^s$$
be the trivial fiber bundle, where $j_2$ is inclusion on to the second factor and $pr_1$ is projection on to the first factor. Then we have the following commutative diagram
$$
\xymatrix{
H^1(\mathbb{R}P^s \times \mathbb{R}P^s) \ar[d]^{j_1^*} & H^1(\mathbb{R}P^s) \ar[l]^{\hspace*{5mm}pr_1^*} \ar[ld]^{\hspace*{3mm}\pi^*}\\
H^1(\mathbb{R}H_{s,s}).& }
$$
Applying theorem \ref{thm4.1} to the fiber bundles given by $\pi$ and $pr_1$ and using theorem \ref{thm4.2},  we get $\pi^*(a_2)=a$ and $pr_1^*(a_2)=a_2$. Further by commutativity of the diagram, we get $j_1^*(a_2)=j_1^*\big(pr_1^*(a_2)\big)=\pi^*(a_2)=a$.

Let $\mathbb{R}P^{s-1} \stackrel{j}{\hookrightarrow} \mathbb{R}P^s$ be the canonical inclusion. Then we have the following  commutative diagram
$$
\xymatrix{
H^1(\mathbb{R}P^s \times \mathbb{R}P^s) \ar[r]^{A^*} \ar[d]^{j_1^*} & H^1(\mathbb{R}P^s \times \mathbb{R}P^s) \ar[d]^{j_1^*} \ar[r]^{j_2^*} & H^1(\mathbb{R}P^{s}) \ar[d]^{j^*}\\
H^1(\mathbb{R}H_{s,s}) \ar[r]^{A^*} & H^1(\mathbb{R}H_{s,s})  \ar[r]^{i^*} & H^1(\mathbb{R}P^{s-1}).}
$$
If $b_3 \in H^1(\mathbb{R}P^{s-1})$ is the generator, then $j^*(b_2)=b_3$. Again by theorems \ref{thm4.1} and \ref{thm4.2}, we get $i^*(b)=b_3$ and $j_2^*(b_2)=b_2$. The commutativity of the right square shows that $j_1^*(b_2)=b$. The commutativity of the left square gives $A^*\big(j_1^*(a_2)\big)=j_1^*\big(A^*(a_2)\big)$. This implies $A^*(a)= j_1^*(b_2)=b$. Thus the induced map $A^*$ is non-trivial.
\end{proof}

We have a similar result for the complex case.

\begin{proposition}\label{prop5.2}
Let $B: \mathbb{C}H_{s,s} \longrightarrow \mathbb{C}H_{s,s}$ be the free involution given by
$$B\big([z_0,\dots,z_s], [w_0,\dots,w_s]\big) \longmapsto \big([w_0,\dots,w_s], [z_0,\dots,z_s]\big).$$
Then $B^*: H^*(\mathbb{C}H_{s,s}) \longrightarrow H^*(\mathbb{C}H_{s,s})$ is non-trivial.
\end{proposition}
\begin{proof}
The proof is similar to that of the real case and left to the reader.
\end{proof}

Note that using propositions \ref{prop5.1} and \ref{prop5.2}, we can construct free actions of  $(\mathbb{Z}/2)^n$ on $\prod_{i=1}^n \mathbb{R}H_{r_i,s_i}$ and $\prod_{i=1}^n \mathbb{C}H_{r_i,s_i}$ for $1 \leq s_i = r_i$ and $n \geq 2$, such that the induced action on mod 2 cohomology is non-trivial.

\section{Proofs of theorems}
\subsection{The real case}
Let $G= (\mathbb{Z}/2)^r$ act freely on a finite dimensional CW-complex $X\simeq_2 \prod_{i=1}^n \mathbb{R}H_{r_i,s_i}$ with $1 \leq s_i \leq r_i$ for each $1 \leq i \leq n$. Further, assume that the induced action on mod 2 cohomology of $X$ is trivial. Using the K\"{u}nneth formula and lemma \ref{lem2.1}, we get
$$H^*(X) \cong \mathbb{Z}/2[a_1,b_1,\dots, a_n, b_n]/I,$$
where $$I=\langle a_1^{s_1+1}, b_1^{r_1}+a_1b_1^{r_1-1}+ \cdots+a_1^{s_1}b_1^{r_1-s_1}, \dots,a_n^{s_n+1}, b_n^{r_n}+a_nb_n^{r_n-1}+ \cdots+a_n^{s_n}b_n^{r_n-s_n}\rangle$$
and $a_1,b_1,\dots,a_n,b_n$ are homogeneous elements of degree one each. Consider the Leray-Serre spectral sequence associated to the Borel fibration
$$X\hookrightarrow X_G \longrightarrow B_G.$$
Then we have $$E_2^{k,l} \cong E_2^{k,0} \otimes E_2^{0,l},$$ where $E_2^{k,0}=H^k(B_G)$ and $E_2^{0,l}=H^0(B_G; \mathcal{H}^l(X))=H^l(X)^G=H^l(X)$ as the induced action on cohomology is trivial. Thus we have
\begin{eqnarray}
E_2^{*,*} & \cong & H^*(B_G) \otimes H^*(X)\nonumber\\
& \cong & \mathbb{Z}/2[\alpha_1,\dots, \alpha_r] \otimes \mathbb{Z}/2[a_1,b_1,\dots, a_n, b_n]/I. \nonumber
\end{eqnarray}

By theorem \ref{thm4.3}, the spectral sequence does not degenerate at the $E_2$ term and hence $d_2:E_2^{0,1} \to E_2^{2,0}$ is non-zero. Let $d_2(1 \otimes a_i)= u_i\otimes 1$ and $d_2(1 \otimes b_i)=v_i\otimes 1$ for $1 \leq i \leq n$ with at least one of them being non-zero. Notice that $d_2$ is completely determined by $d_2(1 \otimes a_i)$ and $d_2(1 \otimes b_i)$ as it is a derivation. Consider the ideal 
$$J=\langle u_1,v_1,\dots,u_n,v_n \rangle$$
in $H^*(B_G)$. Then, using arguments as in \cite{Cusick1}, we have the following lemma.

\begin{lemma}\label{lem6.1}
Let $G=(\mathbb{Z}/2)^r$ act freely on a finite dimensional CW-complex $X\simeq_2 \prod_{i=1}^n \mathbb{R}H_{r_i,s_i}$ with $1 \leq s_i \leq r_i$ for each $1 \leq i \leq n$. Further, assume that the induced action on mod $2$ cohomology of $X$ is trivial and $s_i,r_i \not\equiv 3 \mod 4$ for each $1 \leq i \leq n$. Then the ideal $J$ in $H^*(B_G)$ is invariant under the action of the Steenrod algebra.
\end{lemma}
\begin{proof}
Fix some $1 \leq i \leq n$. Since $s_i \geq 1$, we have $a_i \neq 0$ and hence $b_i^{r_i} \neq 0$. Notice that
\begin{eqnarray*}
b_i^{r_i+1} & = & (a_ib_i^{r_i-1}+\cdots+a_i^{s_i}b_i^{r_i-s_i})b_i\nonumber\\
& = & a_ib_i^{r_i}+\cdots+a_i^{s_i}b_i^{r_i-s_i+1}\nonumber\\
& = & a_i(a_ib_i^{r_i-1}+\cdots+a_i^{s_i}b_i^{r_i-s_i})+a_i^2b_i^{r_i-1}+\cdots+a_i^{s_i}b_i^{r_i-s_i+1}\nonumber\\
& = & a_i^2b_i^{r_i-1}+\cdots+a_i^{s_i}b_i^{r_i-s_i+1}+a_i^{s_i+1}b_i^{r_i-s_i}+a_i^2b_i^{r_i-1}+\cdots+a_i^{s_i}b_i^{r_i-s_i+1}\nonumber\\
& = & 0 \nonumber
\end{eqnarray*}
If $r_i$ is even, then
$$0=d_2(1 \otimes b_i^{r_i+1})=(1 \otimes b_i^{r_i})d_2(1 \otimes b_i)=v_i \otimes b_i^{r_i}.$$
Since the map $-\otimes b_i^{r_i}:E_2^{*,0} \to E_2^{*,r_i}$ is injective, we get $v_i \otimes 1=0$. Recall that $a_i^{s_i+1}=0$. Just as above, if $s_i$ is even, then $u_i \otimes 1=0$. Such $u_i$ and $v_i$ are obviously invariant under the action of the Steenrod algebra.

Next let $r_i=4m+1$ and $v_i \otimes 1 \neq 0$. Notice that $Sq^1(1 \otimes b_i)= 1 \otimes b_i^2$ and $d_2(1 \otimes b_i^{2m})=0$ by the derivation property of $d_2$. Thus $1 \otimes b_i^{2m}$ represents an element in $E_3^{0,2m}$ and we have
\begin{equation}\label{eq1}
0=d_3(1 \otimes b_i^{4m+2})= (1\otimes b_i^{4m})d_3(1 \otimes b_i^2)~\textrm{in}~ E_3^{3,4m}.
\end{equation}
Recall that the transgression operator $d_r:E_r^{0,r-1} \to E_r^{r,0}$ commutes with the Steenrod operations. In other words, the following diagram commute
$$
\xymatrix{
E_2^{0,1} \ar[r]^{d_2} \ar[d]^{Sq^1} & E_2^{2,0} \ar[d]^{Sq^1}\\
E_3^{0,2} \ar[r]^{d_3} & E_3^{3,0}.}
$$
This shows that $d_3(1 \otimes b_i^2)$ is represented by $Sq^1(v_i \otimes 1)$. By equation \eqref{eq1} we have that $(1\otimes b_i^{4m})Sq^1(v_i\otimes 1)$ lies in the image of $d_2:E_2^{1,4m+1} \to E_2^{3,4m}$ and hence 
$$(1\otimes b_i^{4m})Sq^1(v_i \otimes 1)=(1\otimes b_i^{4m})d_2(w)$$
for some $w \in E_2^{1,1}$. Since $1\otimes b_i^{4m} \neq 0$, the map $-\otimes b_i^{4m}:E_2^{*,0} \to E_2^{*,4m}$ is injective, we get $Sq^1(v_i \otimes 1)=d_2(w)$. Let 
$$w= \sum_{j,k} \lambda_{j,k} (\alpha_j  \otimes a_k)+ \sum_{j,k}\mu_{j,k} (\alpha_j \otimes b_k),$$
where $\lambda_{j,k},\mu_{j,k} \in \mathbb{Z}/2$. Then
$$d_2(w)=\sum_{j,k} \lambda_{j,k} (\alpha_ju_k  \otimes 1)+ \sum_{j,k} \mu_{j,k}(\alpha_j v_k \otimes 1) \in J.$$
This shows that $Sq^1(v_i \otimes 1)\in J$. Similarly, if $s_i=4m+1$ and $u_i \otimes 1 \neq 0$, then $Sq^1(u_i \otimes 1) \in J$. Hence $J$ is invariant under the action of the Steenrod algebra.
\end{proof}

For an integer $n$, define
\begin{displaymath}
\eta(n) = \left\{ \begin{array}{ll}
0 & \textrm{if $n$ is even}\\
1 & \textrm{if $n$ is odd.}
\end{array} \right.
\end{displaymath}

\subsection*{Proof of Theorem \ref{thm1.1}}
We first prove (1). As noticed in the proof of lemma \ref{lem6.1}, if some $s_i, r_i$ is even, then the corresponding $u_i, v_i$ is zero. Regard $u_1,v_1,\dots,u_n,v_n \in H^2(B_G)$ as homogeneous polynomials of degree two in $r$ variables. Suppose that they have a non-trivial common zero in $(\mathbb{Z}/2)^r$. Then by proposition \ref{prop4.5}, there is a subgroup inclusion
$$j:\mathbb{Z}/2 \hookrightarrow (\mathbb{Z}/2)^r$$
such that $j^*(u_i)=0=j^*(v_i)$ for each $1 \leq i \leq n$. Restrict the $G$ action on $X$ to $\mathbb{Z}/2$ action on $X$, and consider the Leray-Serre spectral sequence $\{\overline{E}_r^{*,*},\overline{d}_r \}$ associated to the Borel fibration 
$$X \hookrightarrow X_{\mathbb{Z}/2} \longrightarrow B_{\mathbb{Z}/2}.$$ The naturality of the Leray-Serre spectral sequence gives the following commutative diagram
\begin{equation}\label{eq2}
\xymatrix{
E_r^{k,l} \ar[d]^{d_r} \ar[r]^{j^*} & \overline{E}_r^{k,l} \ar[d]^{\overline{d}_r}\\
E_r^{k+r,l-r+1} \ar[r]^{j^*} & \overline{E}_r^{k+r,l-r+1}.}
\end{equation}
Observe that $j^*:E_2^{0,l} \to \overline{E}_2^{0,l}$ is the identity map. This together with the commutative diagram gives $\overline{d}_2(1 \otimes a_i)=0=\overline{d}_2(1 \otimes b_i)$ for each $1 \leq i \leq n$. Hence $\overline{d}_r=0$ for each $r \geq 2$ and 
$$\overline{E}_2^{*,*}=\overline{E}_{\infty}^{*,*}.$$
But we have $H^*(X_G) \cong  \textrm{Tot}E_{\infty}^{*,*}$. This implies that $H^*(X_G)$ is infinite dimensional, a contradiction by proposition \ref{prop4.4}. Hence the system of homogeneous polynomials do not have any non-trivial common zero in $(\mathbb{Z}/2)^r$. Thus by proposition \ref{prop4.7}, $r \leq 2\big(\eta(s_1)+\eta(r_1)+\cdots+\eta(s_n)+\eta(r_n)\big)$.

Next we prove (2). By lemma \ref{lem6.1}, the ideal $J$ is invariant under the action of the Steenrod algebra. By above discussion, the system of homogeneous polynomials $u_1,v_1,\dots,u_n,v_n$ do not have any non-trivial common zero in $(\mathbb{Z}/2)^r$. Thus by proposition \ref{prop4.6}, we have $r \leq (\eta(s_1)+\eta(r_1)+\cdots+\eta(s_n)+\eta(r_n)\big)$. This completes the proof of theorem \ref{thm1.1}. \hfill $\Box$

\subsection{The complex case}
Let $G= (\mathbb{Z}/2)^r$ act freely on a finite dimensional CW-complex $X\simeq_2 \prod_{i=1}^n \mathbb{C}H_{r_i,s_i}$ with $1 \leq s_i \leq r_i$ for each $1 \leq i \leq n$. And suppose that the induced action on mod 2 cohomology of $X$ is trivial. Using lemma \ref{lem2.1}, we get
$$H^*(X) \cong \mathbb{Z}/2[g_1,h_1,\dots,g_n,h_n]/K,$$
where $$K=\langle g_1^{s_1+1},h_1^{r_1}+g_1h_1^{r_1-1}+\cdots+g_1^{s_1}h_1^{r_1-s_1},\dots,g_n^{s_n+1}, h_n^{r_n}+g_nh_n^{r_n-1}+\cdots+g_n^{s_n}h_n^{r_n-s_n}\rangle$$
and $g_1,h_1,\dots,g_n,h_n$ are all homogeneous elements of degree two each. As in the real case, we have
$$E_2^{*,*} \cong \mathbb{Z}/2[\alpha_1,\dots, \alpha_r] \otimes \mathbb{Z}/2[g_1,h_1,\dots,g_n,h_n]/K.$$

Again by theorem \ref{thm4.3}, the spectral sequence does not degenerate at the $E_2$ term. Notice that $d_r=0$ for all even $r$. In particular $d_2=0$ and hence $d_3:E_3^{0,2} \to E_3^{3,0}$ must be non-zero. Let $d_3(1 \otimes g_i)= x_i\otimes 1$ and $d_3(1 \otimes h_i)=y_i\otimes 1$ for $1 \leq i \leq n$ with at least one of them being non-zero. Consider the ideal 
$$L=\langle x_1,y_1,\dots,x_n,y_n \rangle$$
in $H^*(B_G)$. Then we have the following lemma.

\begin{lemma}\label{lem6.2}
Let $G=(\mathbb{Z}/2)^r$ act freely on a finite dimensional CW-complex $X\simeq_2 \prod_{i=1}^n \mathbb{C}H_{r_i,s_i}$ with $1 \leq s_i \leq r_i$ for each $1 \leq i \leq n$. Suppose that the induced action on mod $2$ cohomology of $X$ is trivial and $s_i,r_i \not\equiv 3 \mod 4$ for each $1 \leq i \leq n$. Then the ideal $L$ in $H^*(B_G)$ is invariant under the action of the Steenrod algebra.
\end{lemma}
\begin{proof}
We describe the proof briefly as it is similar to the proof of lemma 6.1. Fix some $1 \leq i \leq n$. Since $s_i \geq 1$, we have $h_i^{r_i} \neq 0$. Notice that $h_i^{r_i+1}=0$. If $r_i$ is even, then $$0=d_3(1 \otimes h_i^{r_i+1})=y_i \otimes h_i^{r_i}.$$ Since the map $-\otimes h_i^{r_i}:E_3^{*,0} \to E_3^{*,r_i}$ is injective, we get $y_i \otimes 1=0$. Since $g_i^{s_i+1}=0$, it follows that if $s_i$ is even, then $x_i \otimes 1=0$. Such $x_i$ and $y_i$ are obviously invariant under the action of the Steenrod algebra.

Let $r_i=4m+1$ and $y_i \otimes 1 \neq 0$. Notice that $Sq^1(1 \otimes h_i)=0$ and the following diagram commute
$$
\xymatrix{
E_3^{0,2} \ar[r]^{d_3} \ar[d]^{Sq^1} & E_3^{3,0} \ar[d]^{Sq^1}\\
E_4^{0,3} \ar[r]^{d_4} & E_4^{4,0}.}
$$
Since $d_4=0$, the commutativity of the diagram shows that $0=Sq^1(y_i \otimes 1) \in L$.

Next we have $Sq^2(1 \otimes h_i)=1 \otimes h_i^2$. Further $d_3(1 \otimes h_i^{2m})=0$ and hence $1 \otimes h_i^{2m}$ represents an element in $E_5^{0,4m}$. Since $h_i^{4m+2}=0$, we have
\begin{equation}\label{eq3}
0=d_5(1 \otimes h_i^{4m+2})= (1\otimes h_i^{4m})d_5(1 \otimes h_i^2)~\textrm{in}~ E_5^{5,4m}.
\end{equation}
Consider the following commutative diagram
$$
\xymatrix{
E_3^{0,2} \ar[r]^{d_3} \ar[d]^{Sq^2} & E_3^{3,0} \ar[d]^{Sq^2}\\
E_5^{0,4} \ar[r]^{d_5} & E_5^{5,0}.}
$$
This shows that $d_5(1 \otimes h_i^2)$ is represented by $Sq^2(y_i \otimes 1)$. By equation \eqref{eq3} we have that $(1\otimes h_i^{4m})Sq^2(y_i\otimes 1)$ lies in the image of $d_3:E_3^{2,4m+2} \to E_3^{5,4m}$ and hence 
$$(1\otimes h_i^{4m})Sq^2(y_i \otimes 1)=(1\otimes h_i^{4m})d_3(z)$$
for some $z \in E_3^{2,2}$. Since $1\otimes h_i^{4m} \neq 0$, the map $-\otimes h_i^{4m}:E_3^{*,0} \to E_3^{*,4m}$ is injective, we get $Sq^2(y_i \otimes 1)=d_3(z)$. Let
$$z= \sum_{j,k,l}\lambda_{j,k,l}(\alpha_j\alpha_k \otimes g_l)+ \sum_{j,k,l}\mu_{j,k,l}(\alpha_j\alpha_k \otimes h_l),$$ where $\lambda_{j,k,l}, \mu_{j,k,l} \in \mathbb{Z}/2.$ Then
$$d_3(z)=\sum_{j,k,l}\lambda_{j,k,l}(\alpha_j\alpha_kx_l \otimes 1)+ \sum_{j,k,l}\mu_{j,k,l}(\alpha_j\alpha_ky_l \otimes 1) \in L.$$
This shows that $Sq^2(y_i \otimes 1) \in L$. Similarly, if $s_i=4m+1$ and $x_i \otimes 1 \neq 0$, then $Sq^1(x_i \otimes 1), Sq^2(x_i \otimes 1) \in L$. Hence $L$ is invariant under the action of the Steenrod algebra.
\end{proof}

\subsection*{Proof of Theorem \ref{thm1.2}}
If some $s_i, r_i$ is even, then the corresponding $x_i,y_i$ is zero. Suppose $x_1,y_1,\dots,x_n,y_n$ have a non-trivial common zero in $(\mathbb{Z}/2)^r$. Then by proposition \ref{prop4.5} there is a subgroup inclusion 
$$j:\mathbb{Z}/2 \hookrightarrow (\mathbb{Z}/2)^r$$
such that $j^*(x_i)=0=j^*(y_i)$ for each $1 \leq i \leq n$. Restrict the $G$ action on $X$ to $\mathbb{Z}/2$ action on $X$, and consider the Leray-Serre spectral sequence $\{\overline{E}_r^{*,*},\overline{d}_r \}$ associated to the Borel fibration $$X \hookrightarrow X_{\mathbb{Z}/2} \to B_{\mathbb{Z}/2}.$$ Observe that $j^*:E_3^{0,l} \to \overline{E}_3^{0,l}$ is the identity map. This together with the commutative diagram \eqref{eq2} gives $\overline{d}_3(1 \otimes g_i)=0=\overline{d}_3(1 \otimes h_i)$ for each $1 \leq i \leq n$. Hence $\overline{d}_r=0$ for each $r \geq 2$ and $\overline{E}_2^{*,*}=\overline{E}_{\infty}^{*,*}$. This gives a contradiction by proposition \ref{prop4.4}. Hence the system of homogeneous polynomials do not have any non-trivial common zero in $(\mathbb{Z}/2)^r$. Thus by proposition \ref{prop4.7}, $r \leq 3\big(\eta(s_1)+\eta(r_1)+\cdots+\eta(s_n)+\eta(r_n)\big)$. This proves (1).

By Lemma \ref{lem6.2}, the ideal $L$ is invariant under the action of the Steenrod algebra. By above discussion, the system of homogeneous polynomials $x_1,y_1,\dots,x_n,y_n$ do not have any non-trivial common zero in $(\mathbb{Z}/2)^r$. Thus by proposition \ref{prop4.6}, we have $r \leq (\eta(s_1)+\eta(r_1)+\cdots+\eta(s_n)+\eta(r_n)\big)$. This proves theorem \ref{thm1.2} (2). \hfill $\Box$
\vspace*{2mm}

Restricting to actions of elementary abelian 2-groups on $\prod_{i=1}^n \mathbb{C}H_{r_i,s_i}$ for which the induced action on mod 2 cohomology is trivial,  we obtain the following corollary.

\begin{corollary}
Let $1 \leq s_i < r_i$ for each $1 \leq i \leq n$. Then
$$\frk_2 \Big( \prod_{i=1}^n \mathbb{C}H_{r_i,s_i} \Big)=\eta(s_1)+\eta(r_1)+\cdots+\eta(s_n)+\eta(r_n)$$
whenever $s_i \equiv 1 \mod 4$ and $r_i \equiv 0,2 \mod 4$.
\end{corollary}
\begin{proof}
Theorem 1.2 (2) gives the upper bound. In section \ref{sec3}, we constructed a $\mathbb{Z}/2$ action on $\mathbb{C}H_{r_i,s_i}$ when $s_i$ is odd and $r_i$ is even, for which the induced action on $H^*\big(\mathbb{C}H_{r_i,s_i}\big)$ is trivial. The products of these actions on $\prod_{i=1}^n \mathbb{C}H_{r_i,s_i} $ achieve the desired bound when $s_i \equiv 1 \mod 4$ and $r_i \equiv 0,2 \mod 4$.
\end{proof}

\section{Some concluding remarks}
We conclude with the following remarks on our results. Adem and Yal\c c\i n asked the following question in \cite[p.70]{Adem3}.
\begin{question}
If $(\mathbb{Z}/2)^r$ acts freely on a finite CW-complex $X$ with mod 2 cohomology generated by one dimensional classes, does it follows that $r \leq 2 \dim H_1\big(X; \mathbb{Z}/2\big)$?
\end{question}
Let $(\mathbb{Z}/2)^r$ act freely on a finite CW-complex $X \simeq_2 \prod_{i=1}^n \mathbb{R}H_{r_i,s_i}$ such that the induced action on mod 2 cohomology is trivial. If $1 \leq s_i < r_i$ for each $1\leq i \leq n$, then $$H_1\big(\mathbb{R}H_{r_i,s_i}; \mathbb{Z}/2\big)=\mathbb{Z}/2 \oplus \mathbb{Z}/2$$ and hence $\dim H_1\big(X; \mathbb{Z}/2\big)=2n$. Therefore, by theorem \ref{thm1.1}(1), $$r \leq 2\big(\eta(s_1)+\eta(r_1)+\cdots+\eta(s_n)+\eta(r_n)\big) \leq 2(2n)=2 \dim H_1\big(X; \mathbb{Z}/2\big).$$ Thus the above question has a positive answer for $X \simeq_2 \prod_{i=1}^n \mathbb{R}H_{r_i,s_i}$.

Let $X$ be as in theorems \ref{thm1.1} and \ref{thm1.2}. Let $s_i$ be even and $r_i$ be odd for each $1 \leq i \leq n$. If $X \simeq_2 \prod_{i=1}^n \mathbb{R}H_{r_i,s_i}$, then the Euler characteristic $\chi(X)=1$. Similarly, if $X \simeq_2 \prod_{i=1}^n \mathbb{C}H_{r_i,s_i}$, then $\chi(X)$ is odd. Hence no elementary abelian 2-group can act freely on $X$ and our theorems are weak in this case.

It is well known that, if a closed smooth manifold does not bound mod 2, then it does not admit any free involution. It was shown in \cite{Khare} that, $\mathbb{R}H_{r,s}$ does not bound for $s=2k+1$ and $r=2^{\beta}(2l+1)$ if and only if one of the following holds
\begin{itemize}
\item $\beta \geq 2$ and $k \geq 1$
\item $\beta=1$, $l+1=2^{\delta}(2t+1)$ and $k \geq 2^{\delta +1}-1$.
\end{itemize}
Thus $\mathbb{R}H_{r,s}$ does not admit any free involution in these cases.
\bigskip

\noindent \textbf{Acknowledgement.}  A part of this work was done when the author was visiting the International Centre for Theoretical Physics (ICTP) in Trieste during March to May of 2013. The author is grateful to the ICTP for providing financial support and excellent working atmosphere. The author also thanks the Department of Science and Technology of India for support via the INSPIRE Faculty Scheme IFA-11MA-01/2011 and the SERC Fast Track Scheme SR/FTP/MS-027/2010.

\bibliographystyle{amsplain}

\end{document}